\documentclass[12pt]{amsart}
\usepackage{amsmath,amssymb,amsfonts,amsthm}
\usepackage{graphicx}
\usepackage{enumerate}
\usepackage{tikz}
\usepackage[mathscr]{eucal}
\usepackage{float}
\numberwithin{equation}{section}
\newtheorem{theorem}{Theorem}[section]

\newtheorem{lemma}[theorem]{Lemma}
\newtheorem{proposition}[theorem]{Proposition}

\newtheorem{claim}[theorem]{Claim}
\newtheorem{conjecture}[theorem]{Conjecture}
\newtheorem{observation}[theorem]{Observation}
\numberwithin{equation}{section}

 \makeatletter
      \def\@setcopyright{}
      \def\serieslogo@{}
      \makeatother

\begin{document}
\bibliographystyle{plain}

\title{Integer sequences with conjectured relation with certain graph parameters of the family of linear Jaco graphs.}
\author{Johan Kok} 
\address{Independent Mathematics Researcher, City of Tshwane, South Africa.}
\email{$^{\dag}$jacotype@gmail.com}
\date{}
\begin{abstract}
This experimental study presents some interesting conjectured relations between some integer sequences and certain graph parameters of the family of linear Jaco graphs $J_n(x)$ where $n = 1,2,3,\dots$. It appears that $\textit{Golden ratio}$-like floor function terms play an important role in the analysis of the graph structural properties of the family of linear Jaco graphs. The experimental methodology to obtain the conjectures is indeed trivial. However, it is the author's view that the proofs or disproofs of the conjectures may be challenging.
\end{abstract}
\maketitle
\section{Introduction}
It is assumed that the reader is familiar with the basic terminology and notions of graph theory, see {\cite{bondy, harary, west}. In 2014 the authors Kok et al. introduced the family of linear Jaco graphs. See \cite{kok1, kok2} and the references thereto. Unless stated otherwise, all graphs in this experimental study will be linear Jaco graphs. 

The celebrated Abel-Ruffini theorem (the year 1824) states that in general, $P(x) = 0$ where $P(x)$ is a polynomial of order five or higher cannot be solved by radicals. Over the years which followed various techniques were developed to solve such polynomial equations by approximation. Surprisingly, in the year 2025 Wildberger et al. \cite{wild} published a (in the author's view) remarkable general solution to all polynomial equations. In \cite{wild} various references are made to graph theoretical results which highlight a cross-over between the boundaries of graph theory and modern algebra and perhaps with mathematical analysis. The aforesaid motivates this study in that studying mathematics for the sake of mathematics is a worthy cause.

In \cite{assous} another cross-over between graph theory and modern algebra is mentioned by Assous et al. It was noted that posets of minimal type are related to the notion of linear Jaco graphs. Also, in \cite {kok4} four conjectures are presented all of which state a relation between an integer sequence and some graph parameter in respect of linear Jaco graphs. Conjectures offer an opportunity to develop techniques of proof or disproof. Furthermore, the attempt to conjecture integers sequences as inherent results offers an opportunity to distinguish between families of graphs for which a given parameter has a closed formula or has a \textit{sequential} formula. Note that a sequential formula can be by recursion through the size of a graph in the family say, for graphs $G_k$, $k = 1,2,3,\dots, n$ or through iteration say, $i$-iterations within the given graph $G$ of size $n$ and $i = 1,2,3,\dots,k$. The motivation for this study does not lie in possible applications \textit{per se} but rather in a new research avenue it could ignite. However, since many integer sequences have real world applications a graph theoretical approach to some of these application could be unveiled.

Recall that a sequence of non-negative integers say $\textbf{d} = (d_1,d_2,d_3,\dots,d_n)$ is said to be graphical if there exists a simple graph $G$ which has degree sequence \textbf{d}. Put differently, \textbf{d} is graphical in respect of the vertex degrees if there exists a simple graph $G$ which has degree sequence \textbf{d}. Note that \textbf{d} could be either a random or a definable sequence. A definable sequence can only be generated from a closed or recursive formula. A generalization can be investigated. That is, if $p$ is a graph parameter (independence number, domination number, chromatic number, hub number and so on) then a definable sequence \textbf{p} = $(p_1,p_2,p_3,\dots,p_n)$ is $p$-graphical if a family of simple graphs exists say $\mathcal{F} = \{G_1, G_2, G_3,\dots,G_n\}$ such that,
\begin{center}
$p(G_1) = p_1$, $p(G_2) = p_2$, $p(G_3) = p_3,\cdots,p(G_n) = p_n$.
\end{center}
\section{On linear Jaco graphs}
Readers are referred to \cite{kok1, kok2} for the formal definitions of $J_\infty(x)$ and $J_n(x)$. In this paper the underlying graph of linear Jaco graphs will be considered. For the latter purpose we recall the construction method as explained in \cite{kok3}.  

\textit{Construction method to obtain $J_n(x)$:}\cite{kok3} 

(a) Let the empty graph (edgeless graph) of order $n \geq 1$ be denoted by $\mathfrak{N}_n$. Let $X = \{v_i: i = 4,5,6,\dots,n\}$

and let $G_2 = P_3 \cup \mathfrak{N}_{n-3}$ where $V(\mathfrak{N}_{n-3}) = X$.

(b) By normal consecutive step-count for $i = 3,4,5,\dots, n-1$ do as follows: 

To obtain $G_i$ add the edges $v_iv_{i+1}$, $v_iv_{i+2}$,$\dots, v_iv_{i+t}$ with $t$ a maximum such that, $deg_{G_i}(v_i) \leq i$.

(c) After completion of step-count $i = n-1$ label the resultant graph $J_n(x)$.
Note that $J_n(x)$ is a simple connected undirected graph.

Recall from \cite{kok3} that for a given $J_n(x)$ the open neighborhood of $v_i$ in a given by:
$N(v_i) = N^-(v_i) \cup N^+(v_i)$ where $N^-(v_i) = \{v_j:j < i$ and the edge $v_jv_i$ exists$\}$ and $N^+(v_i) = \{v_j:j > i$ and the edge $v_iv_j$ exists$\}$.
Also let $t_1(v_i) = |N^-(v_i)|$ and $t_2(v_i) = |N^+(v_i)|$ and let $t^\star_2(v_i) = max\{t_2(v_j):v_j \in V(J_n(x))$ in sufficiently large $J_n(x):n > j\}$. The aforesaid yields $t^\star_2(v_i) = i - t_1(v_i)$.

To identify conjectured integer sequences from linear Jaco graphs the next table of verified data is presented. For $J_n(x)$ the $\Delta$-set (or Jaconian set) is the subset of vertices which have maximum degree in $J_n(x)$. The size (number of edges) is denoted by $\varepsilon(J_n(x))$ and the domination number by $\gamma(J_n(x))$. 
\begin{center}
Table 1. Adapted extract from \cite{kok2}.
\end{center}
\begin{center}
\begin{tabular}{|c|c|c|c|c|c|c|c|}
\hline
$i\in{\Bbb{N}}$&$t_1(v_i)$&$t^\star_2(v_i)$&$\varepsilon(J_n(x))$&$\Delta$-set&$\gamma(J_n(x))$&$diam(J_n(x))$&$\Delta(J_n(x))$\\
\hline
1=$f_2$&0&1&0&$\{v_1\}$&0&0&0\\
\hline
2=$f_3$&1&1&1&$\{v_1,v_2\}$&1&1&1\\
\hline
3=$f_4$&1&2&2&$\{v_2\}$&1&2&2\\
\hline
4&1&3&3&$\{v_2,v_3\}$&2&3&2\\
\hline
5=$f_5$&2&3&5&$\{v_3\}$&2&3&3\\
\hline
6&2&4&7&$\{v_3,v_4,v_5\}$&2&4&3\\
\hline
7&3&4&10&$\{v_4,v_5\}$&2&4&4\\
\hline
8=$f_6$&3&5&13&$\{v_5\}$&2&4&5\\
\hline
9&3&6&16&$\{v_5,v_6,v_7\}$&2&5&5\\
\hline
10&4&6&20&$\{v_6,v_7\}$&2&5&6\\
\hline
11&4&7&24&$\{v_7\}$&2&5&7\\
\hline
12&4&8&28&$\{v_7,v_8\}$&3&5&7\\
\hline
13=$f_7$&5&8&33&$\{v_8\}$&3&5&8\\
\hline
14&5&9&38&$\{v_8,v_9,v_{10}\}$&3&6&8\\
\hline
15&6&9&44&$\{v_9,v_{10}\}$&3&6&9\\
\hline
16&6&10&50&$\{v_{10}\}$&3&6&10\\
\hline
17&6&11&56&$\{v_{10},v_{11}\}$&3&6&10\\
\hline
18&7&11&63&$\{v_{11}\}$&3&6&11\\
\hline
19&7&12&70&$\{v_{11},v_{12},v_{13}\}$&3&6&11\\
\hline
20&8&12&78&$\{v_{12},v_{13}\}$&3&6&12\\
\hline
21&8&13&86&$\{v_{13}\}$&3&6&13\\
\hline
22&8&14&94&$\{v_{13},v_{14},v_{15}\}$&3&7&13\\
\hline
23&9&14&103&$\{v_{14},v_{15}\}$&3&7&14\\
\hline
24&9&15&112&$\{v_{15}\}$&3&7&15\\
\hline
25&9&16&121&$\{v_{15},v_{16}\}$&3&7&15\\
\hline
\end{tabular}
\end{center}

\begin{center}
\begin{tabular}{|c|c|c|c|c|c|c|c|}
\hline
$i\in{\Bbb{N}}$&$t_1(v_i)$&$t^\star_2(v_i)$&$\varepsilon(J_n(x))$&$\Delta$-set&$\gamma(J_n(x))$&$diam(J_n(x))$&$\Delta(J_n(x))$\\
\hline
26&10&16&131&$\{v_{16}\}$&3&7&16\\
\hline
27&10&17&141&$\{v_{16},v_{17},v_{18}\}$&3&7&16\\
\hline
28&11&17&152&$\{v_{17},v_{18}\}$&3&7&17\\
\hline
29&11&18&163&$\{v_{18}\}$&3&7&18\\
\hline
30&11&19&174&$\{v_{18},v_{19},v_{20}\}$&3&7&18\\
\hline
31&12&19&186&$\{v_{19},v_{20}\}$&3&7&19\\
\hline
32&12&20&198&$\{v_{20}\}$&3&7&20\\
\hline
\end{tabular}
\end{center}

The subscript $i$ means that the infinite linear Jaco graph $J_\infty{x}$ is under consideration and $i$ identifies a specific vertex $v_i$ of $J_\infty{x}$. The subscript $n$ means that the finite linear Jaco graph $J_n(x)$ is considered where $n = 1,2,3,\dots$ All sequences mentioned in the conjectures are found in \cite{sloane}.

\subsection{Size of $J_n(x)$, $n = 1,2,3,\dots$:} Recall that Wagner \cite{kok1, kok3} identified that in accordance with sequence A060144 we have the result $t_1(v_i) = \lfloor \frac{2(i+1)}{3+\sqrt{5}}\rfloor$, $i = 1,2,3,\dots$ holds. Put differently, the sequence A060144 is $p$-graphical where $p(G) = t_1(v_i)$, $i = 1,2,3,\dots$ and $\mathcal{F} = \{G:G=J_n(x), n = 1,2,3,\dots\}$ or simply $J_\infty(x)$.
\begin{proposition}\label{prop1}
With the minor adaption of ignoring $n=0$ the size of $J_n(x)$, $n = 1,2,3,\dots$ increases in accordance with sequence A183137 with formula
\begin{center}
$\varepsilon(n) = \sum\limits_{i=1}^{n}\lfloor \frac{2(i+1)}{3+\sqrt{5}}\rfloor$, $n = 1,2,3,\dots$
\end{center}

Furthermore, it implies that sequence A183137 is $p$-graphical where $p(G) = \varepsilon(G)$ and $\mathcal{F} = \{G:G=J_n(x), n = 1,2,3,\dots\}$.
\end{proposition}
\begin{proof}
The validity of $t_1(v_i) = \lfloor \frac{2(i+1)}{3+\sqrt{5}}\rfloor$ read together with the construction of a finite linear Jaco graph implies immediately that
\begin{center}
$\varepsilon(J_n(x)) = \sum\limits_{i=1}^{n}t_1(v_i)$, $n = 1,2,3,\dots$ 
\end{center}
That settles the proof.
\end{proof}
\subsection{Maximum degree of $J_n(x)$, $n = 1,2,3,\dots$:} With regards to Table 1 column 8 and besides $\Delta(J_1(x)) = 0$ (to be excluded) the maximum degree of $J_n(x)$, $n = 2,3,4,\dots$ is conjectured below.
\begin{conjecture}\label{conj1}
For the finite linear Jaco graphs $J_n(x)$, $ n = 2,3,4,\dots$ the respective maximum degrees $\Delta(J_n(x))$ are yielded by the sequence A319433 with closed formula $\Delta(J_n(x)) = \lfloor \frac{2(n+2)}{1+ \sqrt{5}}\rfloor -1$, $n = 2,3,4,\dots$

Furthermore, it implies that sequence A319433 is $p$-graphical where $p(G) = \Delta(G)$ and $\mathcal{F} = \{G:G=J_n(x), n = 1,2,3,\dots\}$.
\end{conjecture}
\textit{Motivation:} Considering Table 1 column 5 and read together with the construction of a linear Jaco graph it follows as a trivial fact that for a given $n$ the $\Delta$-set of $J_n(x)$ contains the vertex $v_{n-t_1(v_n)}$. Hence, $\Delta(J_n(x) = t_1(v_n) + t_1(v_{v_{n-t_1(v_n)}})$. What remains to be proven is that:
\begin{center}
$\lfloor \frac{2(n+1)}{3+\sqrt{5}}\rfloor + \lfloor \frac{2(q+1)}{3+\sqrt{5}}\rfloor = \lfloor \frac{2(n+2)}{1+ \sqrt{5}}\rfloor -1$ where $q = n-\lfloor \frac{2(n+1)}{3+\sqrt{5}}\rfloor.$
\end{center}
 
From Table 1 column 8 it also follows that besides that $\Delta(J_1(x)) = 0$ (to be excluded) the maximum degree values which appear uniquely (appear once) i.e. $\Delta = 1,4,6,9,12,14,17,\dots$ are conjectured below.
\begin{conjecture}\label{conj2}
For $t = 1,2,3,\dots$ the values given by sequence A003622 i.e. the Wythoff compound sequence AA: $a(t) = \lfloor t \times (\frac{1+\sqrt{5}}{2})^2\rfloor -1$ are the unique maximum degrees of corresponding linear Jaco graphs.

Furthermore, it implies that sequence A003622 is $p$-graphical where $p(G) = \Delta(G)_{unique}$ and $\mathcal{F} \in \{G:G=J_n(x), n = 1,2,3,\dots\}$.
\end{conjecture}
\begin{conjecture}\label{conj3}
The order $n \geq 3$ of the linear Jaco graphs for which a maximum degree $\Delta(J_n(x))$ appears once (unique) is given by the sequence A035336: $a(t) = 2 \times \lfloor \frac{t(1+\sqrt{5})}{2}\rfloor +(t-1)$, $t = 3,4,5,\dots$

Furthermore, it implies that sequence A035336 is $p$-graphical where $p(G) = |V((G)|_{unique~\Delta(G)}$ and $\mathcal{F} \in \{G:G=J_n(x), n = 1,2,3,\dots\}$.
\end{conjecture}
Read Conjectures \ref{conj2} and \ref{conj3} in conjunction with columns $5$ and $8$ in Table 1 to understand the subtle difference between the conjectured results. Next we conjecture for which order $n \geq 1$ a linear Jaco graph contains a $\Delta$-set (or Jaconian set) $X$ with $|X| = 1$, $|X| = 2$ or $|X| = 3$, respectively.

\begin{conjecture}\label{conj4}
The order $n \geq 3$ of the linear Jaco graphs for which:

(a) The $\Delta$-set $X$ has $|X| = 1$ is given by the sequence $a(t) = w(t+1) - 2$, $t = 1,2,3,\dots$ and where sequence A001950 is the upper Wythoff sequence (a Beatty sequence) given by
\begin{center}
$w(1) = \lfloor \frac{t(1+\sqrt{5})^2}{4}\rfloor$, $t = 1,2,3,\dots$
\end{center}

(b) The $\Delta$-set $X$ has $|X| = 2$ is given by the sequence A057843: $a(t) = \lfloor \frac{t(1+\sqrt{5})^2}{4}\rfloor -3$, $ t = 2,3,4,\dots$

(c) The $\Delta$-set $X$ has $|X| = 3$ is given by the adapted sequence A134859: $a(t) = 2 \times \lfloor \frac{t(1+\sqrt{5})^2}{4}\rfloor - (t+2)$, $ t = 2,3,4,\dots$

Furthermore, it implies that sequences A001950, A057843, A134859 are $p$-graphical where
\begin{equation*} 
p(G) =
\begin{cases}
|X|=1;\\
|X|=2;\\
|X|=3.\\
\end{cases}
\end{equation*} 

respectively, and $\mathcal{F} \in \{G:G=J_n(x), n = 1,2,3,\dots\}$.
\end{conjecture}
From Table 1 column 5 it follows that the vertices which do not serve as a unique vertex with maximum degree hence, the $v_i^{,s}$ such that a $\Delta$-set $\{v_i\}$ does not exist are given by $i = 4,6,9,12,14,17,19$. By adding $"1"$ as an artificial first term the sequence $i = "1",4,6,9,12,14,17,19$ underpins the next conjecture.
\begin{conjecture}\label{conj5}
The vertices $v_i$ which do not yield a $\Delta$-set $X$, $|X| = 1$ in some finite linear Jaco graph $J_n(x)$ are yielded by sequence A003622 i.e. the Wythoff compound sequence AA with close formula: $i = \lfloor \frac{i\times(1+ \sqrt{5})^2}{4}\rfloor - 1$, $i = 2,3,4,\dots$

Furthermore, it implies that sequences A003622 (excluding the first term) is $p$-graphical where $p(G) = i$ and $\{v_i\} \neq \Delta$-set of $G$ and $\mathcal{F} \in \{G:G=J_n(x), n = 1,2,3,\dots\}$. 
\end{conjecture}

\subsection{Minimum domination set:} In general there exist numerous $\gamma$-sets for $J_n(x)$, $n \geq 2$. Generally, a vertex $v_i \in V(J_n(x))$ has a neighbor $v_j$ with maximum subscript $j >i$. The vertex $v_j$ is called the \textit{upper neighbor} of $v_i$. For $i\geq 2$ a vertex $v_i \in V(J_n(x))$ also has a neighbor $v_r$ with minimum subscript $r < i$. The vertex $v_r$ is called the \textit{lower neighbor} of $v_i$. 

\begin{figure}[h]
\centering
\begin{tikzpicture}[scale=1.0]
\tikzstyle{every node}=[draw,shape=circle];
\node (v0) at (0:3) [] {$v_1$};
\node (v1) at (45:3) [] {$v_2$};
\node (v2) at (2*45:3) [] {$v_3$};
\node (v3) at (3*45:3) [] {$v_4$};
\node (v4) at (4*45:3) [] {$v_5$};
\node (v5) at (5*45:3) [] {$v_6$};
\node (v6) at (6*45:3) [] {$v_7$};
\node (v7) at (7*45:3) [] {$v_8$};
\draw (v0) -- (v1)
(v1) -- (v2)
(v2) -- (v3)
(v2) -- (v4)
(v3) -- (v4)
(v3) -- (v5)
(v3) -- (v6)
(v4) -- (v5)
(v4) -- (v6)
(v4) -- (v7)
(v5) -- (v6)
(v5) -- (v7)
(v6) -- (v7)
;
\end{tikzpicture}
\caption{Jaco graph $J_8(x)$}\label{fig1}
\end{figure}
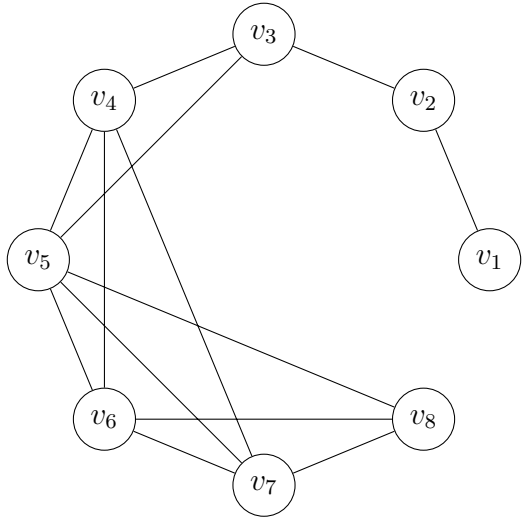

From Figure \ref{fig1} it is easy to see that $diam(J_8(x)) = 4$ on path $P_{diam}(J_8(x)) = v_1v_2v_3v_5v_8$ and the path $P_d(J_8(x)) = v_1v_2v_3v_4v_7v_8$ has length $5$. For the two paths \textit{per se} it follows that $\gamma(P_{diam}(J_8(x))) = \gamma(P_d(J_8(x))) = 2$. Also, $\gamma(J_8(x)) = 2$. It is equally easy to verify that for $J_{15}(x)$ we have $diam(J_{15}(x)) = 6$ on $P_{diam}(J_{15}(x)) = v_1v_2v_3v_5v_8v_{13}v_{15}$ and the path $P_d(J_{15}(x)) = v_1v_2v_3v_4v_7v_{11}v_{15}$ has length $6$. For the two paths \textit{per se} it follows that $\gamma(P_{diam}(J_{15}(x))) = \gamma(P_d(J_{15}(x))) = 3$. Also, $\gamma(J_{15}(x)) = 3$. However, the $\gamma$-set for $P_{diam}(J_{15}(x))$ i.e. $\{v_3,v_8,v_{15}\}$ does not dominate vertex $v_4$. Therefore, it does not dominate $J_{15}(x)$. On the other hand the $\gamma$-set for $P_d(J_{15}(x))$ i.e. $\{v_2,v_7,v_{15}\}$ is a minimum dominating set of $J_{15}(x)$.
\begin{observation}\label{obs1}
For any finite linear Jaco graph $J_n(x)$, $n \geq 2$ there exists a pair of vertices i.e. $v_1,v_n$ for which a minimal $(v_1,v_n)$-path (not necessarily a $diam$-path) i.e. $P_d(J_n(x))$ exists such that a $\gamma$-set of $P_d(J_n(x))$ is a $\gamma$-set of $J_n(x)$. We call the path $P_d(J_n(x))$ the \textit{primary minimal dom-path}. 
\end{observation}
Recall that for a path $P_n$, $n = 2,3,4,\dots$ the domination number is given by $\gamma(P_n) = \lceil \frac{n}{3}\rceil$. The proof of the aforesaid result provides the key argument. In a sufficiently large linear Jaco graph $J_n(x)$ the vertex $v_2$ dominates vertices $v_1,v_3$. Now proceed to vertex $v_4$. To minimize the choice of dominating vertices select the \textit{upper neighbor} of $v_4$ i.e. vertex $v_7$ as well as the \textit{upper neighbor} of $v_7$ i.e. vertex $v_{11}$. Clearly the set $\{v_2,v_4\}$ dominates the linear Jaco graphs $J_4(x)$, $J_5(x)$, $J_6(x)$ and $J_7(x)$. The set $\{v_2,v_7\}$ dominates $J_{n}(x)$, $7\leq n\leq 11$. Now proceed to vertex $v_{12}$ and repeat the reasoning to minimize the choice of dominating vertices. Clearly, the aforesaid constitutes a heuristic method to yield a primary minimal dom-path for any linear Jaco graph $J_n(x)$, $n \geq 1$.
\begin{observation}\label{obs2}
A secondary minimal dom-path $Q_d(J_n(x))$ exists. The procedure remains the same by beginning at vertex $v_n$. For example, for $J_8(x)$ the path $Q_d(J_8(x)) = v_8v_5v_3v_2v_1$ is the secondary minimal dom-path. For $J_{15}(x)$ the path $Q_d(J_{15}(x)) = v_{15}v_9v_6v_5v_3v_2v_1$ is the secondary minimal dom-path.
\end{observation}
Henceforth, $P_d(G)$ will be abbreviated as $P_d$ if the context of graph $G$ is clear.
\begin{conjecture}\label{conj6}
For any linear Jaco graph $J_n(x)$, $n \geq 1$ the length of a $diam$-path and a primary minimal dom-path $P_d$ satisfy
\begin{center}
$|P_d|-|diam(J_n(x))| \leq 1$.
\end{center}
\end{conjecture}
\begin{claim}\label{clm1}
A $\gamma$-set of a  $diam$-path \textit{per se} is not necessarily a $\gamma$-set of $J_n(x)$.
\end{claim}

We begin with a lemma which formalizes Observation \ref{obs1}. 
\begin{lemma}\label{lem1}
If a general graph $G$ has a pair of vertices say $v_i, v_j$ for which some minimal dom-path $P_d = (v_i,v_j)$-path exists (not necessarily a diam-path) with a $\Delta$-set $X$ of the path $P_d$ \textit{per se} and $\bigcup\limits_{v_i \in X}N[v_i] = V(G)$ then $\gamma(G) = |X| = \gamma(P_d)$.
\end{lemma}
\begin{proof}
Since $\bigcup\limits_{v_i \in X}N[v_i] = V(G)$ then obviously $\gamma(G) \leq |X|$. In respect of $V(P_d)$ \textit{per se} no less vertices will suffice unless a vertex $v_j \in V(G)\backslash X$ such that at least two vertices say $v_t,v_q \in X$ at distance $d(v_t,v_q) > 2$ can be replaced by some vertex $v_s \in V(G)\backslash X$ which dominates $G$. If the latter is possible then $P_d$ is not such shortest path. This is a contradiction. Thus, the result is settled.
\end{proof} 
It is a trivial fact that $\gamma(J_1(x)) = 0$, $\gamma(J_2(x)) = \gamma(J_3(x)) = 1$, $\gamma(J_4(x)) = 2$. For $J_n(x)$, $n \geq 5$ consider a sufficiently large $J_m(x)$, $m \geq n$ to find the path:
\begin{center}
$P_d = \underbrace{v_1v_2v_3}_{3-string}\underbrace{v_4v_7v_{11}}_{3-string}\underbrace{v_{12}v_{20}v_{32}}_{3-string}\underbrace{v_{(q_1=32+1)}v_{(q_2=33+t^\star_2(v_{33}))}v_{(q_3= q_2+t^\star_2(v_{q_2}))}}_{3-string}\cdots\underbrace{v_{s_1}v_{s_2}v_{s_3}}_{3-string}$
\end{center}
such that, $s_3$ is a minimum to permit: 
\begin{center}
$v_n \in \{v_{s_1},v_{(s_1+1)},v_{(s_1+2)},\dots,v_{s_2},v_{(s_2+1)},v_{(s_2+2)},\dots,v_{s_3}\}$.
\end{center}
From the pattern followed to define $P_d$ it is easy to see that:

Case 1: If $v_n \in \{v_{s_1},v_{(s_1+1)},v_{(s_1+2)},\dots,v_{s_2}\}$ then
\begin{center}
$X_1 = \{v_2,,v_7,v_{20},v_{54}, \dots, v_n\}$ is a $\gamma$-set of path
\end{center}
\begin{center}
$P_d = \underbrace{v_1v_2v_3}_{3-string}\underbrace{v_4v_7v_{11}}_{3-string}\underbrace{v_{12}v_{20}v_{32}}_{3-string}\underbrace{v_{(q_1=32+1)}v_{(q_2=33+t^\star_2(v_{33}))}v_{(q_3= q_2+t^\star_2(v_{q_2}))}}_{3-string}\cdots\underbrace{v_{s_1}v_n}$.
\end{center}
By Lemma \ref{lem1} it follows that $X_1$ is a $\gamma$-set of $J_n(x)$.

Case 2: If $v_n \in \{v_{(s_2+1)},v_{(s_2+2)},\dots,v_{s_3}\}$ then
\begin{center}
$X_2 = \{v_2,,v_7,v_{20},v_{54}, \dots, v_{s_2}\}$ is a $\gamma$-set of path
\end{center}
\begin{center}
$P_d = \underbrace{v_1v_2v_3}_{3-string}\underbrace{v_4v_7v_{11}}_{3-string}\underbrace{v_{12}v_{20}v_{32}}_{3-string}\underbrace{v_{(q_1=32+1)}v_{(q_2=33+t^\star_2(v_{33}))}v_{(q_3= q_2+t^\star_2(v_{q_2}))}}_{3-string}\cdots\underbrace{v_{s_1}v_{s_2}v_n}$.
\end{center}
By Lemma \ref{lem1} it follows that $X_2$ is a $\gamma$-set of $J_n(x)$.

With the artificial first term "1" added hence, by considering the sequence $i = "1",2,7,20,54,\dots$ the next conjecture follows.
\begin{conjecture}\label{conj7}
The vertex subscripts of a $\gamma$-set $X = "\{v_1\}"\cup \{v_2,v_7,v_{20},v_{54},\dots\}$ of the infinite linear Jaco graph $J_\infty(x)$ is given by the sequence A000149: $a(t) = \lfloor e^t\rfloor$, $t = 0,1,2,\dots$ where $e \approx 2.71828$ is the Euler number (or Napier's constant).

Furthermore, it implies that sequence A000149 is $p$-graphical where $p(G) = \{i:j$ the subscript of $v_j \in X$ with $X$ some $\gamma$-set of $G\}$ and $\mathcal{F} = \{G:G=J_n(x), n = 1,2,3,\dots\}$.
\end{conjecture}
\section{Avenue for research}
Consider the cycle $C_5 = v_1v_2v_3v_4v_5v_1$. Obtain the Petersen graph $\mathcal{P}$ by adding a pendent vertex $u_i$ to each of the vertices $v_i$, $i = 1,2,3,4,5$ and thereafter, add the edges $u_1u_3$, $u_1u_4$, $u_2u_4$, $u_2u_5$, $u_3u_5$. A minimal dom-path is given by $P_d(\mathcal{P}) = v_1u_1u_4u_2u_5u_3v_3$ in that $X = \{v_1,u_3,u_4\}$ is a $\gamma$-set of both $P_d(\mathcal{P})$ and the Petersen graph $\mathcal{P}$.\\
\textit{Illustration:} For the path $P_4 = v_1v_2v_3v_4$ it is easy to see that the set of distinct $\gamma$-sets is
\begin{center}
$\mathcal{X} = \{\{v_1,v_3\},\{v_1,v_4\},\{v_2,v_3\},\{v_2,v_4\}\}$.
\end{center}
Lemma \ref{lem1} can be strengthened by the next existence theorem.
\begin{theorem}
Any connected graph $G$ of order $n \geq 1$ has a minimal dom-path $P_d(G)$ such that a $\gamma$-set $X$ of $P_d(G)$ has $\bigcup\limits_{v_i \in X}N[v_i] = V(G)$ hence, $\gamma(P_d(G)) = \gamma(G)$.
\end{theorem}
\begin{proof}
Consider the set $\mathcal{X} = \{X_j:X_j$ a distinct $\gamma$-set of $G\}$. Since $G$ is connected it is inherently possible to find a shortest path $P(X_j)$ in $G$ which contains all the vertices of any $X_j \in \mathcal{X}$. If $\gamma(P(X_j)) = \gamma(G)$ then $\gamma(P(X_j))$ serves as $P_{d_j}(G)$. Furthermore, if $\gamma(G) \neq \gamma(P(X_j))$, ($\gamma(G) < \gamma(P(X_j))$ is not possible), it is inherently possible to minimally increase the length of $P(X_j)$ to obtain a minimal path $P_{d_j}(G)$. Let $P_d(G) = P(X_i)$ where $|P(X_i)| = min|P(X_j)|$, $P(X_j) \in \mathcal{X}$. That settles the result. 
\end{proof}
\section{Conclusion} 
It is not the claim that the integer sequences which have a relation with some graph parameter of the family of linear Jaco graphs have been exhausted. It fact, it is suggested that an interesting and worthy avenue for research is open.

The methods of proof or disproof of the collection of conjectures presented in this experimental study are left for other sophisticated minds to pursue.
\section*{Dedication}
This experimental study is dedicated to late Mr Fransloo Vorster (nicknamed "Fifi") who was the teacher of mathematics of the author at Bellville High School, Western Cape Province, South Africa. The dedication serves to acknowledge the remarkable way in which Mr Vorster could explain the wonders of High School Mathematics. May his soul rest in peace. 
\section*{Acknowledgment}
The author would like to thank the anonymous referees for their constructive comments, which helped to improve on the elegance of this paper.

\end{document}